\documentclass[12pt]{amsart}
\usepackage{a4wide}
\usepackage{amssymb}
\usepackage{amsthm}
\usepackage{amsmath}
\usepackage{amscd}
\usepackage{verbatim}
\usepackage[all]{xy}

\numberwithin{equation}{section}

\theoremstyle{plain}
\newtheorem{theorem}{Theorem}[section]
\newtheorem{corollary}[theorem]{Corollary}
\newtheorem{lemma}[theorem]{Lemma}
\newtheorem{proposition}[theorem]{Proposition}

\theoremstyle{definition}
\newtheorem{definition}[theorem]{Definition}
\newtheorem{remark}[theorem]{Remark}

\theoremstyle{remark}

\newcommand{\OO}{\mathcal O}
\newcommand{\A}{\mathbb{A}}
\newcommand{\R}{\mathbb{R}}

\newcommand{\Q}{\mathbb{Q}}
\newcommand{\Z}{\mathbb{Z}}

\newcommand{\C}{\mathbb{C}}

\renewcommand{\H}{\mathbb{H}}

\newcommand{\zxz}[4]{\begin{pmatrix} #1 & #2 \\ #3 & #4 \end{pmatrix}}
\newcommand{\abcd}{\zxz{a}{b}{c}{d}}

\newcommand{\kzxz}[4]{\left(\begin{smallmatrix} #1 & #2 \\ #3 & #4\end{smallmatrix}\right) }

\newcommand{\vol}{\operatorname{vol}}
\newcommand{\tr}{\operatorname{tr}}

\newcommand{\Spin}{\operatorname{Spin}}

\newcommand{\Hom}{\operatorname{Hom}}
\newcommand{\Aut}{\operatorname{Aut}}

\newcommand{\GL}{\operatorname{GL}}
\newcommand{\SO}{\operatorname{SO}}

\newcommand{\cha}{\operatorname{char}}

\newcommand{\kay}{k}
\newcommand{\Gspin}{\operatorname{GSpin}}

\newcommand{\ff}{\hbox{if }}
\newcommand{\SL}{\operatorname{SL}}
\newcommand{\gen}{\operatorname{gen}}

\begin{document}

\title[Ternary quadratic forms and representation number]{Ternary quadratic forms and Heegner divisors}

\author[]{}
\dedicatory{}

\address{}
\email{}

\thanks{}


\date{\today}
\author[Tuoping Du]{Tuoping Du}
\address{Department of Mathematics, Northwest University, Xi'an, 710127 ,  P.R. China}
\email{dtpnju@gmail.com}
\thanks{dtpnju@gmail.com}
\maketitle

\begin{abstract}
In this paper,  I use Siegel-Weil formula and Kudla matching principle to prove some interesting identities between representation number (of ternary quadratic space ) and the degree of  Heegner divisors.

\end{abstract}

\maketitle


\section{Introduction}
\label{sect: introduction}

Kudla found in \cite{KuIntegral} that Fourier coefficients of some Eisenstein series have geometric meaning.
He also obtained a useful matching principle which is an identity between two genus theta series from two different quadratic spaces as both are special values of the same Eisenstein series. This simple identity connects two different quadratic spaces, also give arithmetic and geometric interpretation of the coefficients. In this paper, I use this principle to prove some new identities on ternary quadratic spaces, and relate the representation numbers for lattices in definite space and degree of Heegner divisors in Shimura curves.

Let $D$ be a square  free positive integer, and let $B=B(D)$ be the unique
 quaternion algebra of discriminant $D$ over $\Q$, i.e., $B$ is
 ramified at a finite prime $p$ if and only if $p|D$.  The reduced
 norm, denoted by $\det$ in this paper.   Let $$V(D)=\{x\in B(D)\mid \tr(x)=0 \}$$ with the restriction  quadratic form $\det$, where $\tr$ is the reduced trace.

 For a positive integer $N$ prime to
 $D$, let $L_D(N)=\OO_D(N) \bigcap V(D) $, where $\OO_D(N)$  is an Eichler order in $B$ of conductor $N$. We can view $L=(L_D(N), \det)$ as
 an even integral lattice in $V(D)$. The quaternion $B$ is definite if and only if $D$ has odd
 number of prime factors.

When $V(D)$ is positive definite($D$ has odd number of prime factors),
there is a very interesting but hard  question to compute the
 representation number (for a positive integer $m$)
\begin{equation}
r_L(m)=|\{ x \in L_D(N):\, \det x =
m\}|.
\end{equation}
In general, it is very hard to compute this number. However, there is a available method to compute the average over the genus $\gen(L)$, which we denote by
\begin{equation}\label{genusr}
r_{D, N}(m) =r_{\gen(L)}(m) =\bigg(\sum_{L_1 \in \gen(L)}
\frac{1}{|\Aut(L_1)|}\bigg)^{-1} \sum_{L_1 \in \gen(L)}
\frac{r_{L_1}(m)}{|\Aut(L_1)|}.
\end{equation}
It is a  product of local densities, Siegel studied
in 1930's \cite{siegel formula}. These densities  are computable (see
\cite{YaDensity}, \cite{Yatwo}).

When  $V(D)$ is indefinite, the representation number does not make sense anymore since a number can be represented by infinitely many points. The number is related to the degree of Heegner divisors in Shimura curves for this case.

Fix an embedding  $ i: B(D)\hookrightarrow M_2(\R)$ such that $B(D)^\times $ is invariant under the automorphism  $x \mapsto x^l= {}^tx^{-1}$ of $\GL_2(\R)$. Let   $\Gamma_0^D(N) =\OO_D(N)^1$ be the  group of (reduced) norm $1$ elements in $\OO_D(N)$ and   let  $X_0^D(N) = \Gamma_0^D(N) \backslash \H$ be the associated Shimura curve. For a positive integer $m$, let $Z_{D, N}(m)$ be Heegner divisors in $X_0^D(N)$  associated to the lattice $L_{D}(N)$ which we will introduce in Section \ref{sect:definite and indefinite}. The degree \begin{equation}\deg Z_{D, N}(m)= \sum_{x\in \Gamma \setminus L_{m}} \frac{1}{\mid \Gamma_{x} \mid},\end{equation} see  \cite[3.27]{Fu} .

 Let $\Omega =\frac{1}{2\pi} y^{-2} dx \wedge dy$ be the normalized differential form on $X_0^D(N)$, and let
 \begin{equation}
 \vol(X_0^D(N), \Omega) = \int_{X_0^D(N)} \Omega
 \end{equation}
 be the  volume of $X_0^D(N)$ with respect to $\Omega$, which is a  positive rational number (see (\ref{eq:volume})).

Finally,  we define the normalized degree (when $D$ has even number of prime factors)
\begin{equation}\label{hdeg}
r_{D, N}(m)= \frac{\deg Z_{D,  N}(m)}{ \vol(X_0^D(N), \Omega)},
\end{equation}
the same notation as definite case. In this paper, when $D$ has odd number of prime factors, $r_{D, N}(m)$ denote average representation number (\ref{genusr}), when $D$ has even number of prime factors $r_{D, N}(m)$ denote normalized degree (\ref{hdeg}).

From Kudla's matching principle, we know Fourier coefficients of some theta series are associated to $r_{D, N}$.
In the paper \cite{DuYang}, we proved some interesting identities between average representation number for quaternion algebras. We also proved the result on the degrees of Hecke correspondences on Shimura curves. The main idea of this paper is the same as  \cite{DuYang}, bur more complicated because of the odd dimensional Weil representation(see Section 2). I will prove analogues result Theorem \ref{theo1.1} on the three dimension quadratic space. For the indefinite space $V(D)$, we give an interesting identity Corollary \ref{coro1.2} for Heegner divisors.

\begin{theorem} \label{theo1.1} Let $D$ be a square-free positive
integer, let $p \ne q$ be two
different primes not dividing $D$, and let $N$ be a positive integer
prime to $Dpq$. Then
$$
-\frac{2}{q-1} r_{Dp, N}(m) + \frac{q+1}{q-1} r_{Dp, Nq}(m)
 =-\frac{2}{p-1} r_{Dq, N}(m) + \frac{p+1}{p-1} r_{Dq, Np}(m)
$$
for every positive integer $m$.
\end{theorem}

In the work \cite{BerkovichJagy}, Alexander Berkovich and William C. Jagy obtained a similar result
\begin{equation}s(p^{2}n)-ps(n)=48\sum_{\tilde{f}\in TG_{1,p}}\frac{R_{\tilde{f}}(n)}{|\Aut(\tilde{f})|}-96\sum_{\tilde{f}\in TG_{2,p}}\frac{R_{\tilde{f}}(n)}{|\Aut(\tilde{f})|},\end{equation}
where $s(n)$ is the representation number of $x^{2}+y^{2}+z^{2}=n$. Indeed, two terms of the right side are average representation number on genus. Their work is concrete and they compute local density directly. My result is general, and I don't compute the local density but use Kudla matching.

 Katsurada and Schulze-Pillot studied the action of Hecke operator on genus theta functions (\cite{KSP}) and obtained some very interesting formulas between different genus theta functions. Essentially, these identities are also Kudla's matching principle's application as they pointed in \cite[Section 6]{KSP}.  In our paper (\cite{DuYang}) and this paper, $r_{D, N}(m)$'s are Fourier coefficients of some genus theta functions of weight $2$ and $3/2$ respectively, which are different from theirs.
When the space $V(D)$ is indefinite, the Theorem \ref{theo1.1} is geometric. For the Heegner divisors,
\begin{corollary}\label{coro1.2}
Let the notations be as above, then $$-2\deg Z_{Dp, N}(m)+\deg Z_{Dp, Nq}(m)=-2\deg Z_{Dq, N}(m)+\deg Z_{Dq, Np}(m).$$
\end{corollary}
In the work \cite{Fu}, J.Funke Showed more details of Heegner divisors and weight $3/2$ modular forms.
The above results are relations between the same kind number $r_{D, N}$. The next theorem is a relation between different kinds numbers,
\begin{theorem} \label{theo1.3}  Let $D >1$ be a square-free positive
integer, let $p \nmid D$ be a prime, and   let $N$ be a positive integer
prime to $Dp$. Then
$$
r_{Dp, N}(m)= -\frac{2}{p-1} r_{D, N}(m) + \frac{p+1}{p-1} r_{D, Np}(m).
$$
\end{theorem}

\begin{remark}
Notice that, when $D=1$, there is no Siegel-Weil formula for space $V(1)$. For this space, there is a regularized Siegel-Weil formula. The coefficients of regularized theta integral are not related to our problem, so we assume $D >1$ in the Theorem $\ref{theo1.3}$. Indeed, Theorem \ref{theo1.1} and \ref{theo1.3} are different for $D$ which has even or odd number prime factors.
\end{remark}

From this theorem, we know representation number $r_{D, N}(m)$ are closely related to degree of Heegner divisors. The representation number has a geometric interpretation for the degree of Heegner divisors, and they could be computed by each other. From the arithmetic geometry point of view, this is natural.

 This paper is organized as below. In Section \ref{sect:Kudla}, recall the  Weil representation and Kudla's matching principle in general case. In Section  \ref{sect:matching}, I prove  some explicit local matchings  between Schwartz functions on two space $V^{ra}$, $V^{sp}$  and also construct global matching pairs. In Section \ref{sect:definite and indefinite}, I discuss the definite case and study coefficients of theta integral. I will show a fact that $r_{D, N}(m)$ is the $m$-Fourier coefficient of some theta integral. Then introduce the Shimura curve and Heegner divisors in the later part. In Section \ref{result}, we found Fourier coefficients of  the theta integral are connected  to the degree of Heegner divisors for indefinite space. At the end I prove Theorems \ref{theo1.1}, \ref{theo1.3} and Corollary \ref{coro1.2}.

\section{Preliminaries} \label{sect:Kudla}

Let $(V, Q)$ be a nondegenerate quadratic space over $\mathbb{Q}$, and let
$ G=\SL_{2}$. Fix the unramified canonical additive character $$\psi:\quad \mathbb{A}/\mathbb{Q}\rightarrow \C^{\times}, \quad \psi_{\infty}(x) = e^{2\pi ix}.$$ There is a nontrivial twofold extension
$$ 1 \longrightarrow  \mu_{2} \longrightarrow \widetilde{G} \longrightarrow G \longrightarrow 1.$$
 We identify
$\widetilde{G}(\mathbb{A})=\SL_{2}(\mathbb{A})\times \{{\pm 1}\}$,
where multiplication on the right is given by
$$[g_{1}, \epsilon_{1}][g_{2}, \epsilon_{2}]=[g_{1}g_{2},\epsilon_{1}\epsilon_{2}c(g_{1}, g_{2})],$$
for the cocycle as in
\cite{KuIntegral}. For subgroup $P$ of $G$, we denote $\widetilde{P}$ the full inverse image in $\widetilde{G}$.

In particular, we have subgroups

\vspace{0.2cm}

$ N^{\prime}(\mathbb{A})=\{n=[n(b), 1] | b \in \mathbb{A}\}$,
\begin{displaymath}
n(b)=
\left(
  \begin{array}{cc}
    1 & b \\
     & 1 \\
  \end{array}
\right)
\end{displaymath}
and

$\widetilde{M}(\mathbb{A})= \{m=[m(a), \varepsilon] | a\in \mathbb{A}^{\times}, \varepsilon=\pm1\}$,
\begin{displaymath}
m(a)=
\left(
  \begin{array}{cc}
    a &  \\
     & a^{-1} \\
  \end{array}
\right)
\end{displaymath}

Let $\chi=\chi_{V}$ be the quadratic character of ${\mathbb{A}}^{\times} / \mathbb{Q}^{\times}$ associated with $V$ defined by

\begin{center}
$\chi(x)=(x, (-1)^{m(m-1)/2}\det(V))$,
\end{center}
where $m=\dim(V)$ and $\det(V) \in \mathbb{Q}^{\times}/\mathbb{Q}^{\times,2}$ is the determinant of the matrix for the quadratic form $Q$ on $V$.
$\chi$ determines a character $\chi^{\psi}$ on $\widetilde{M}(\A)$ by $$ \chi^{\psi}([m(a), \varepsilon])=\varepsilon \chi(a)\gamma(a, \psi),$$ where $\gamma(a, \psi)$ is the Weil index.

The group $\widetilde{G}(\mathbb{A})$ acts on the Schwartz space $S(V(\mathbb{A}))$ via the Weil representation $\omega=\omega_{\psi}$ determined by our fixed additive character ${\psi}$ of $\mathbb{A}/\mathbb{Q}$, and this action  commutes with the linear action of $O(V)(\mathbb{A})$.
The $\widetilde{\SL_2}(\A)$-action is determined by (see for example \cite{KuSplit})
\begin{eqnarray}\label{weilrep}
&\omega((n(b),1))\varphi(x)=\psi(bQ(x))\varphi(x), \nonumber\\
&\omega((m(a),\varepsilon)\varphi(x)=\chi^{\psi}(a, \varepsilon) \mid a\mid^{\frac{m}2} \varphi(ax), \\
&\omega((w, 1))\varphi= \gamma(V)^{-1}  \widehat{\varphi} = \gamma(V)^{-1}\int_{V(\A)}\varphi(y)\psi((x, y))dy ,\nonumber
\end{eqnarray}
where for $a \in \A^\times$, $b \in \A$
\begin{center}
$n(b)=
\left(
  \begin{array}{cc}
    1 & b \\
     & 1 \\
  \end{array}
\right),
m(a)=
\left(
  \begin{array}{cc}
    a &  \\
     & a^{-1} \\
  \end{array}
\right)
,
w=\left(
  \begin{array}{cc}
     & 1 \\
     -1&  \\
  \end{array}
\right),$
\end{center}
 $dy$ is the Haar measure on $V(\A)$ self-dual with respect to $\psi ((x, y))$, and $\gamma(V)=\prod_{p} \gamma(V_p)=1$, where $\gamma(V_p)$ is a $8$-th root of unity  associated to  the local Weil representation at place $p$ (local Weil index).  Let  $P=NM$  be the standard Borel subgroup of $\SL_2$, where $N$ and $M$ are subgroups of $n(b)$ and $m(a)$ respectively.

For $s \in \mathbb{C}$, let $I(s, \chi)$ be the principal series representation of $\widetilde{G}_{\mathbb{A}}$ consisting of smooth functions $\Phi(s)$ on $\widetilde{G}(\mathbb{A})$ such that

\begin{equation} \label{eq:1.1}
\Phi(nm(a)g^{\prime}, s) =
\left\{
\begin{array}{cc}
\chi^{\psi}(m(a))|a|^{s+1}\Phi(g^{\prime}, s) &\mbox {if m is odd,}\\
\chi(m(a))|a|^{s+1}\Phi(g^{\prime}, s) &\mbox {if m is even.}
\end{array}\right.
\end{equation}

There is a $\widetilde{G}(\mathbb{A})$ intertwining map
\begin{equation}
\lambda=\lambda_{V} : S(V(\mathbb{A})) \rightarrow I(s_{0}, \chi),
\end{equation}

\begin{center}
$\lambda(\varphi)(g^{\prime})= \omega(g^{\prime})(0)$,
\end{center}
where $\omega$ is the Weil representation of the group $\widetilde{G}(\mathbb{A})$.
Let $\widetilde{K}_{\infty}\widetilde{K}$ be the full inverse image of
$SO(2)\times \SL_{2}(\hat{\mathbb{Z}})$ in $\widetilde{G}(\mathbb{A})$.
A section $\Phi(s)\in I(s, \chi)$ is called
standard if its restriction to $\widetilde{K}_{\infty}\widetilde{K}$ is
independent of s. By Iwasawa decomposition  $\widetilde{G}_{\mathbb{A}}=\widetilde{M}_{\mathbb{A}}N^{\prime }_{\mathbb{A}} \widetilde{K}_{\infty}\widetilde{K}$, the function
$\lambda(\varphi)\in I(s_{0}, \chi)$ has a unique extension to a
standard section $\Phi(s)\in I(s, \chi)$, where
$\Phi(s_{0})=\lambda(\varphi)$.

For $g^{\prime} \in \widetilde{G}(\mathbb{A})$, $h \in O(V)(\mathbb{A})$ and $\varphi \in S(V)(\mathbb{A})$
, the theta series is defined as

$$\theta(g^{\prime}, h; \varphi)=\sum_{x\in V(\mathbb{Q})} \omega(g^{\prime})\varphi(h^{-1}x).$$
For an algebraic group $G$ over $\Q$, we write  $[G]=G(\Q) \backslash G(\A)$.
The theta  integral
\begin{equation}
I(g^{\prime}, \varphi)=\frac{1}{\vol([O(V)])} \int_{[O(V)]} \theta(g^{\prime}, h, \varphi)dh
\end{equation} is an automorphic form on $[\widetilde{G}]$ if it is convergent.
The Eisenstein series is given by
\begin{equation}
E(g^{\prime}, s; \Phi)=E(g^{\prime}, s; \varphi)= \sum_{\gamma\in \widetilde{P}(\mathbb{Q})\setminus \widetilde{G}(\mathbb{Q})}
\Phi(\gamma g^{\prime}, s).
\end{equation}
\begin{theorem} (Siegel- Weil formula) \label{theo:Siegel-Weil}
Assume that V is anisotropic or that $\dim(V)-r>2$, where r is the Witt index of V, so that the theta  integra is absolutely convergent. Then \\

$(1)$ $E(g^{\prime}, s; \Phi)$ is holomorphic at the point $s_{0}=m/2-1$, where $m=\dim(V)$, and
\begin{center}
$E(g^{\prime}, s_{0}; \Phi)=\kappa I(g^{\prime}, \varphi)$,
\end{center}
 where $\kappa=2$ when $m\leq2$ and $\kappa=1$ otherwise.\\

(2)If $m>1$, then
\begin{center}
$E(g^{\prime}, s_{0}; \Phi)=\kappa I(g^{\prime}, \varphi)=\frac{\kappa}{2}\int_{[SO(V)]} \theta(g^{\prime}, h; \varphi)dh $,
\end{center}
where dh is Tamagawa measure on $SO(V)(\mathbb{A})$.
\end{theorem}

Let $V^{(1)}, V^{(2)}$ be two quadratic spaces with the same dimension and the same quadratic character $\chi$. There is a following diagram
\begin{equation}
\setlength{\unitlength}{1mm}
\begin{picture}(60, 20)
\linethickness{1pt}
\put(0,18){$S(V^{(1)}(\mathbb{A}))$}
\put(0,0){$S(V^{(2)}(\mathbb{A}))$}
\put(18,18){ \vector(3,-1){25}}
\put(18,0){ \vector(3,1){25}}
\thicklines
\put(45,8){$I(s_{0},\chi)$}
\put(25,16){$\lambda_{V^{(1)}}$}
\put(25,6){$\lambda_{V^{(2)}}$}
\end{picture}.
\end{equation}

There are analogous local maps
\begin{center}
$\lambda_{p}: S(V_{p})\rightarrow I_{p}(s_{0}, \chi_{p})$.
\end{center} Following Kudla \cite{KuIntegral},  we make the following definition.

\begin{definition} For an prime $p \le \infty$, $\varphi_p^{(i)} \in S(V_p^{(i)})$, $i=1, 2$,  are said to be matching if
$$
\lambda_{V_p^{(1)}} (\varphi_p^{(1)}) = \lambda_{V_p^{(2)}}(\varphi_p^{(2)}).
$$
$\varphi^{(i)}=\prod_p \varphi_p^{(i)} \in S(V^{(i)}(\A))$ are said to be matching if they match at each prime $p$.
\end{definition}

By the Siegel-Weil formula, we have the following Kudla matching principle (\cite[Section 4]{KuIntegral}):  Under the assumption of Theorem  \ref{theo:Siegel-Weil} for both $V^{(1)}$ and $V^{(2)}$,
one has  matching  pair $(\varphi^{(1)}, \varphi^{(2)})$ and the following identity:
\begin{equation} \label{eq:matching}
I(g^{\prime}, \varphi^{(1)})=I(g^{\prime}, \varphi^{(2)}).
\end{equation}
This implies that their Fourier coefficients are equal, which we use in this paper. Comparing coefficients of both sides, I obtain main results of this paper.

\section{Matchings on quadratic spaces}  \label{sect:matching}

 There are two quaternion algebras over a local field $\Q_p$, the  matrix algebra  $B^{sp}=M_2(\Q_p)$ (split quaternion) and the division quaternion $B^{ra}$ (ramified  quaternion). Let $V^{sp}=B_{0}^{sp}$ or $V^{ra}=B_{0}^{ra}$ be the associated three dimensional quadratic space with reduced norm,  $$B_{0}^{sp}=\{x \in B^{sp} \mid \tr(x)=0\}, \quad B_{0}^{ra}=\{x \in B^{ra} \mid \tr(x)=0\},$$ where $\tr$ is the reduced trace. Both spaces have the same quadratic character $\chi_{p}=(x, -1)_{p}$. We have $\widetilde{G}(\Q_p)$-intertwining operators
$$
\lambda: S(V) \rightarrow I(\frac{1}{2}, \chi_{p} ),$$
$$\lambda(\varphi)(g^{\prime})=\omega(g^{\prime}) \varphi(0),$$
where $I(\frac{1}{2}, \chi_{p} )$ are the same for associated space $V^{sp}$ and $V^{ra}$.

We will use superscript $sp$  and $ra$ to indicate the  association with $B_{0}^{sp}$ and $B_{0}^{ra}$ respectively. It is known (\cite{KuIntegral}) that $\lambda^{sp}$ is surjective. So every function $\varphi^{ra}$ in $S(V^{ra})$ has matching element. The purpose of this section is to give some explicit matching pairs and to obtain some interesting local and global identities.

\subsection{ The finite prime case $p <\infty$}  We assume $p <\infty$ in this subsection. Let $\OO_{B^{ra}}$ be the maximal order in $B^{ra}$, which consists of all elements of $B$ whose reduced norm is in $\Z_p$. We don't use the subscript $p$ for simplicity in this subsection.

 Let $\kay$ be the unique unramified quadratic field extension of $\Q_p$,  and let $\OO_\kay=\Z_p + \Z_p u$ be the ring of integers of $\kay$ with $u \in \OO_\kay^\times$. Fix one optimal embedding
 $ \kay  \hookrightarrow  B^{ra}$, then there is a uniformizer $\pi$ of $B$ such that $\pi r = \bar r \pi$ for $r \in \kay$ and $\pi^\iota =-\pi$ and $ \pi^2 =p$ (see \cite{Gross}). Then one has
$$\OO_{B^{ra}} = \OO_\kay + \OO_\kay \pi = \Z_p + \Z_p u + \Z_p \pi + \Z_p u p.
$$

Let $L^{ra}=\OO_{B^{ra}} \bigcap V^{ra}$ and $L^{sp}= M_2(\Z_p) \bigcap V^{sp}$. Then there is a sublattice of $L^{sp}$
$$
L_1^{sp}= \left\{ A =\abcd \in L^{sp}:\,  c \equiv 0 \pmod  p \right\}.
$$
The dual lattices are given by
$$
L^{ra, \sharp} =\OO_{B^{ra}}^{\sharp} \bigcap V^{ra} , \quad  L_1^{sp, \sharp} = \left\{ \abcd \in V^{sp}:\,  a, c, d \in \Z_p,  b \in \frac{1}p \Z_p\right\},
$$
where $\OO_{B^{ra}}^{\sharp}=\pi^{-1} \OO_{B^{ra}}$,
$\pi \in B^{ra}$ is the uniformizer  , and
$$
L^\sharp =\{ x \in  V :\,  (x, L) \subset \Z_p \}.
$$

  We denote
\begin{equation}
\varphi^{ra} =\cha (L^{ra}), \quad  \varphi^{ra, \sharp}= \cha (L^{ra, \sharp}), \quad
\end{equation}
and
\begin{equation} \label{eq:varphi}
\varphi^{sp} =  \cha(L^{sp}),  \quad \varphi_1^{sp} =  \cha(L_1^{sp}) \quad   \hbox{ and }  \varphi_1^{sp, \sharp}= \cha(L_1^{sp, \sharp}).
\end{equation}

So one has isomorphisms
$$
(Z/p)^2 \cong \OO_{B^{ra}}^{ \sharp}/\OO_{B^{ra}}, \quad (Z/p)^2 \cong L^{ra, \sharp}/L^{ra}.
$$

\begin{lemma}\cite[Lemma 14.3]{KRYComp}\label{weilindex}
For the character $\psi$, the Weil index\\
$\gamma(V^{ra})=-1$, $\quad \gamma(V^{sp})=1$, when  $p\neq 2$,\\
and $\gamma(V^{ra})=-\zeta_8^{-1}$, $\quad \gamma(V^{sp})=\zeta_{8}^{-1}$, when $p= 2$.
\end{lemma}

The following are local matching pairs which are need in this paper:
\begin{proposition}  \label{localmatch} Let notations be as above. Then

(1) \quad $ \varphi^{ra} \in S(V^{ra})$ matches with $\frac{-2}{p-1}\varphi^{sp}+\frac{p+1}{p-1}\varphi_1^{sp} \in S(V^{sp})$.

(2) \quad $\varphi^{ra, \sharp} \in S(V^{ra})$ matches with $\frac{2p}{p-1}\varphi^{sp}-\frac{p+1}{p-1}\varphi_1^{sp, \sharp} \in S(V^{sp})$.

\end{proposition}
\begin{proof} (1) \quad  Since $$\SL_2(\Z_p) =K_0(p) \cup N(\Z_p)  w K_0(p), $$ one has
\begin{equation}\label{decomp}
\widetilde{\SL_2(\Z_p)} = \widetilde{K_0(p)} \cup  N^{\prime}(\Z_p) w^{\prime} \widetilde{K_0(p)},
\end{equation}
where $$K_0(p)= \left\{ \abcd \in \SL_{2}(\Z_p):\,  c \equiv 0 \pmod  p\right\}, w^{\prime}=[w, 1].$$
The dimension of $I(\frac{1}{2}, \chi_{p} )^{\widetilde{K_0(p)}}$ is $2$, and $\Phi \in I(\frac{1}{2}, \chi_{p} )^{\widetilde{K_0(p)}}$ is determined by $\Phi(1)$ and $\Phi(w^{\prime})$ from the group decomposition (\ref{decomp}).

 Notice that $K_0(p)$ is generated by $n(b)$ and
$n_-(c) = w^{-1} n(-c) w$, $b \in \Z_p$ and $c \in p \Z_p$, so $\widetilde{K_0(p)}$ is generated by $[n(b), 1]$, $[n_-(c), 1]$ and $[1,\varepsilon ]$, where $\varepsilon= \pm 1$.  Then one can check that $\varphi^{ra}$, $\varphi^{sp}$, $\varphi_1^{sp}$ are all $\widetilde{K_0(p)}$-invariant for the Weil representation. We check
 $\omega([n_-(-c), 1])\varphi^{ra} =\varphi^{ra}$ and leave others to the reader.  One has
$$
\omega^{ra}(w^{\prime})\varphi^{ra}(x) =\gamma(V^{ra}) \varphi^{ra, \sharp}(x) \vol(L^{ra}).
$$
So
$$
\omega^{ra}([n(-c),1]w^{\prime})\varphi^{ra}(x)=\gamma(V^{ra})\vol(L^{ra}) \psi_p(-c \det (x)) \varphi^{ra, \sharp}(x)
=\gamma(V^{ra}) \varphi^{ra, \sharp}(x) \vol(L^{ra}),
$$
i.e.,
$$
\omega^{ra}([n(-c),1]w^{\prime})\varphi^{ra}= \omega^{ra}(w^{\prime})\varphi^{ra}.
$$
Then
$$
\omega^{ra}([n_-(c), 1]) \varphi^{ra} = \omega^{ra}( w^{\prime, -1}) \omega^{ra}([n(-c), 1]w^{\prime})\varphi^{ra} =\varphi^{ra}
.$$
Notice that $ w^{\prime, -1}=[w^{-1}, 1]$  when p is odd , $ w^{\prime, -1}=[w^{-1}, -1]$  when p is 2 \cite{HM}.

It is easy to know
$\omega^{ra}([n(b), 1])\varphi^{ra} =\varphi^{ra},$ and
$\omega^{ra}([1, \varepsilon])\varphi^{ra} =\varphi^{ra}$, so as claimed $\lambda^{ra}(\varphi^{ra}) \in I(\frac{1}{2}, \chi_{p} )^{\widetilde{K_0(p)}}$.

Now we have $\lambda^{ra}(\varphi^{ra}), \lambda^{sp}(\varphi^{sp}), \lambda^{sp}(\varphi_1^{sp}) \in I(\frac{1}{2}, \chi_{p} )^{\widetilde{K_0(p)}}$. Direct calculation gives
\begin{align*}
\lambda^{ra}(\varphi^{ra})(1) &= 1,  \quad \lambda^{ra}(\varphi^{ra})(w^{\prime})= \gamma(V^{ra})^{-1} p^{-1}
\\
\lambda^{sp}(\varphi^{sp})(1)  &= 1,  \quad  \lambda^{sp}(\varphi^{sp})(w^{\prime})  = \gamma(V^{sp})^{-1},
\\
\lambda^{sp}(\varphi_1^{sp})(1)  &= 1,  \quad  \lambda^{sp}(\varphi_1^{sp})(w^{\prime})  = \gamma(V^{sp})^{-1}p^{-1}.
\end{align*}
From Lemma \ref{weilindex}, one has
$$
\lambda^{ra}(\varphi^{ra})= \frac{-2}{p-1} \lambda^{sp}(\varphi^{sp})  + \frac{p+1}{p-1} \lambda^{sp}(\varphi_1^{sp}).
$$
This proves (1). Claim (2) is similar and is left to the reader. One just needs to replace $K_0(p)$ by
$$
K_0^+(p) = \left\{ \abcd \in \SL_2(\Z_p):\,  b \equiv 0 \pmod  p \right\}.
$$

\end{proof}

\subsection{The case $p=\infty$} In this subsection, we consider the case $\Q_p=\R$ and  recall a matching pair given in \cite{KuIntegral}.  Notice that $B^{ra}$ in this case is the Hamilton division algebra, and $V^{ra}$ has signature $(3, 0)$. Let $\varphi_\infty^{ra} (x) =e^{ - 2 \pi \det (x)} \in S(V^{ra})$, then $\varphi_\infty^{ra}$ is of weight $3/2$ in the sense
$$
\omega^{ra}(k_\theta^{\prime}) \varphi_\infty^{ra} = e^{\frac{3}{2} i \theta}   \varphi_\infty^{ra}, \quad k_\theta = \kzxz {\cos\theta} {\sin\theta} {-\sin\theta} {\cos\theta}, \quad k_\theta^{\prime}=[k_\theta, 1]
$$
On the other hand, Kudla constructed a family of  weight $3/2$ Schwartz function $\varphi_\infty^{sp} \in S(V^{sp})$  as follows \cite[Section 4.8]{KuIntegral}. Recall $V^{sp} =\{x\in  M_2(\R) \mid \tr(x)=0\}.$ Given an orthogonal decomposition
\begin{equation} \label{eq:spacedecomposition}
V^{sp} = V^+ \oplus V^-,  \quad x = x^+ + x^-,
\end{equation}
with $V^+$ of signature $(1,0)$ and $V^-$ of signature $(0, 2)$. One defines (Kudla used the notation $\tilde\varphi(x, z)$)
$$
\varphi_\infty^{sp} (x, V^-) = (4\pi (x^+, x^+) -1) e^{ -\pi (x^+, x^+) + \pi (x^-, x^-)}.
$$
Kudla proved the following proposition \cite[Section 4.8]{KuIntegral}.

\begin{proposition} \label{prop3.3}   For any orthogonal decomposition  (\ref{eq:spacedecomposition}), $(\varphi_\infty^{ra}, \varphi_\infty^{sp}(x , V^-))$ is a matching pair, and their (same) image in $I(\frac{1}{2}, \chi_{\infty} )$ is the unique weight $3/2$ section $\Phi_\infty^{\frac{3}{2}}$ given by
$$
\Phi_\infty^{\frac{3}{2}}(n(b) m(a) k_\theta^{\prime})= |a|^{\frac{3}{2}} e^{\frac{3}{2} i \theta}.
$$
\end{proposition}

Because of this matching pair, we will simply write  $\varphi_\infty^{sp}$ for $\varphi_\infty^{sp} (\, ,  V^-)$.

\subsection{Global matching}
 The following global matching result is clear from Kudla's matching principle (\ref{eq:matching}), Propositions \ref{localmatch} and \ref{prop3.3}.

\begin{proposition}  \label{globalmatch} Let $D_1, D_2 >1$ be two square free integers, and  let $V(D_i)$ be the tenary quadratic spaces associated to  the quaternion algebras $B(D_i)$ over $\Q$ (with reduced norm as the quadratic form), $i=1, 2$. Assume that $\varphi^{(i)}=\prod_p \varphi_p^{(i)} \in S(V(D_i)(\A))$ satisfy the following conditions:

(1) \quad When $p =\infty$,   $\varphi_\infty^{(i)}$  is $\varphi_\infty^{sp}$ or $\varphi_\infty^{ra}$ depending on whether $V(D_i)_\infty$ is  split or non-split.

(2) \quad When $p \nmid D_1 D_2 \infty$ or $p | \hbox{gcd}(D_1, D_2)$, we identify $V(D_1)_p = V(D_2)_p$ and take any $\varphi_p^{(1)}= \varphi_p^{(2)} \in S(V(D_1)_p)$.

(3) \quad When $ p|\hbox{lcm}(D_1, D_2)$ but $p\nmid \hbox{gcd}(D_1, D_2)$, one of $V(D_i)_p$ is $V_p^{sp}$ and the other one is $V_p^{ra}$, we take $(\varphi_p^{(1)}, \varphi_p^{(2)})$ to be a matching pair in Propositions \ref{localmatch}.

Then $(\varphi^{(1)}, \varphi^{(2)})$ is a global matching pair, and
$$
I(g^{\prime}, \varphi^{(1)}) = I(g^{\prime}, \varphi^{(2)}), \quad g^{\prime} \in \widetilde{G}(\mathbb{A}).
$$
\end{proposition}
Let $V(D)=\{x \in B(D) \mid \tr(x)=0\}$ be the quadratic space, where  $B(D)$ is quaternion algebra with discriminant $D$. Recall that a Eichler order of conductor $N$ denoted by $\OO_D(N)$,  is an order of $B(D)$ such that
\begin{enumerate}
\item  When $p|D$, $\OO_D(N)_p:=\OO_D(N) \otimes_\Z \Z_p$ is the maximal order in division quaternion algebra $B(D)_p=B_p^{ra}$.

\item When $p\nmid D\infty$,   there is an identification $B(D)_p \cong M_2(\Q_p)$ under which
$$
\OO_D(N)_p:=\left \{ \abcd \in M_2(\Z_p): \,  c \equiv 0 \mod N\right\}.
$$
\end{enumerate}
Now let the lattice $L_{D}(N)= \OO_D(N) \bigcap V(D)$ in the space $V(D)$.  Then we know
\begin{enumerate}
\item  When $p|D$, $L_D(N)_p:=\L_D(N) \otimes_\Z \Z_p$ is a sublattice of the maximal order in $B(D)_p=B_p^{ra}$, denoted by $L^{ra}_{p}$.
\item When $p\nmid D\infty$,
$$
L_D(N)_p:=\left \{ x=\abcd \in M_2(\Z_p): \,  \tr(x)=0, c \equiv 0 \mod N\right\},
$$ denoted by $L^{sp}_{p}$.
\end{enumerate}
 Let $V^{(1)}=V(Dp)$ and $V^{(2)}=V(Dq)$. Define $\varphi^{(1)} =\prod_l \varphi_l^{(1)} \in S(V^{(1)}(\A))$ as follows,
$$
\varphi_l^{(1)} =\begin{cases}
  \varphi_\infty^{ra}(\varphi_\infty^{sp}) &\ff l =\infty, D \ has \ even(odd) \ number\ primes,
  \\
 \cha( L_l^{sp})  &\ff   l \nmid Dpq,
 \\
 \varphi_{l}^{ra}     &\ff l | Dp,
 \\
 \frac{-2}{l-1} \varphi_{l}^{sp} + \frac{l+1}{l-1} \varphi_{l, 1}^{sp}   &\ff  l=q
 \end{cases}
$$
where $\varphi_{l}^{sp}$ , $\varphi_{l, 1}^{sp}$and $\varphi_l^{ra}$ are the functions defined in  (\ref{eq:varphi}) with added subscript $l$. Then one has
$$
\varphi_f^{(1)} = \frac{-2}{q-1} \cha(\widehat{L_{Dp}(N)}) + \frac{q+1}{q-1}\cha(\widehat{L_{Dp}(Nq)}).
$$
So
\begin{equation}
I(\tau, \varphi^{(1)}) = \frac{-2}{q-1} I(\tau, L_{Dp}(N)) + \frac{q+1}{q-1}I(\tau, L_{Dp}(Nq)).
\end{equation}
Let $\varphi^{(2)}$ be defined as  $\varphi^{(1)}$ with  the roles of $p$ and $q$ switched. Then $\varphi^{(1)}$ and $\varphi^{(2)}$ form a matching pair by Proposition \ref{localmatch}. So Proposition \ref{globalmatch} implies
 $$
 I(\tau, \varphi^{(1)}) =I(\tau, \varphi^{(2)}),
 $$
that is
\begin{proposition}\label{pro3.5}
$$\frac{-2}{q-1} I(\tau, L_{Dp}(N)) + \frac{q+1}{q-1}I(\tau, L_{Dp}(Nq))
= \frac{-2}{p-1} I(\tau, L_{Dq}(N)) + \frac{p+1}{p-1}I(\tau, L_{Dq}(Np)).$$
\end{proposition}
Taking  $m$-th Fourier coefficients, we could prove main results in this paper.
In next two sections, we will give arithmetic and geometric interpretations of the theta integrals in some special cases and prove theorems in the introduction.

\section{representations numbers and Heegner divisors}  \label{sect:definite and indefinite}
\subsection{Definite quadratic space and representations numbers}
In this subsection, show a general fact about positive definite quadratic forms for the convenience of the readers. We could see that the Fourier coefficients of theta integral associated to definite quadratic space are closely related to representation number over genus.

Let $(V, Q)$ be a positive definite quadratic space of dimension $m$. Define Gaussian
$$
\varphi_\infty(x) = e^{-2 \pi Q(x)}  \in S(V(\R)).
$$
Then we know
$$
\varphi_\infty(hx) =\varphi_\infty(x), \quad \omega(k_\theta^{'}) \varphi_\infty = e^{\frac{m}2 i \theta} \varphi_\infty
$$
for  $h \in O(V)(\R)$ and $k_\theta \in \SO_2(\R) \subset \SL_2(\R)$.

 For any $\varphi_f \in S(\hat V)$, where $\hat V = V \otimes_\Z \hat\Z$, define the theta kernel
$$
\theta(\tau, h, \varphi_f \varphi_\infty) = v^{-\frac{m}4} \theta(g_\tau^{\prime}, h, \varphi_f \varphi_\infty)
$$
is a holomorphic modular form of weight $\frac{m}2$ for some congruence subgroup. Here $g_\tau = n(u) m(\sqrt v)$ for $\tau =u + i v \in \mathbb H$  and $g_{\tau}^{\prime}=(g_{\tau}, 1)$, $n(u)$ and $m(\sqrt v)$ are introduced in Section 2. So
$$
I(\tau, \varphi_f\varphi_\infty) = v^{-\frac{m}4} I(g_\tau^{\prime}, \varphi_f\varphi_\infty)
$$
is also a modular form of weight $\frac{m}2$.

For an even integral lattice $L$ in $V$, we  let
\begin{equation} \label{eq:new4.1}
\theta(\tau, L) = \theta(\tau,  \cha(\hat L) \varphi_\infty), \quad I(\tau, L)=I(\tau,  \cha(\hat L) \varphi_\infty),
\end{equation}
where $\hat L = L \otimes_\Z \hat\Z$.
Notice that two lattices $L_1$ and $L_2$  in $V$ are in same class if there is $h \in O(V)(\Q)$ such that $hL_1 =L_2$. Two lattices $L_1$ and $L_2$  are in the same genus if they are equivalent locally everywhere, i.e,  there is $h \in O(V) (\hat\Q)$ such that $h L_1=  L_2$. The group $O(V)(\A)$ acts on the set of lattices as follows: $h L =  (h_f \hat L)\cap V$ where $h_f$ is the finite part of $h=h_f h_\infty$.

 Let $\gen(L)$ be the genus of $L$ (the set of all lattices in the same genus of $L$). Then from the above discussion, there is a bijective map
$$
O(V)(\Q) \backslash O(V)(\A)/K(L) O(V)(\R) \cong  \gen(L), \quad [h] \mapsto hL,
$$
where $K(L)$ is the stabilizer subgroup of $\hat L$ in $O(V)(\hat\Q)$.

\begin{proposition} \cite{DuYang}\label{pro4.1}Let
$$
r_L(n) =|\{ x \in  L:\,  Q(x) =n\}|, \quad  r_{\gen(L)}(n) = \left(\sum_{L' \in \gen(L)} \frac{1}{|O(L')|}\right)^{-1} \sum_{L' \in \gen(L)}  \frac{r_{L'}(n)}{|O(L')|},
$$
where $O(L)$ is the stabilizer of $L$ in $O(V)$. Then, for $q =e(\tau)$,

 \begin{equation}\nonumber \begin{split}\theta(\tau, h,  L) =\sum_{n=0}^\infty r_{hL}(n) q^n,\\
I(\tau, L) = \sum_{m=0}^\infty r_{\gen(L)}(n) q^n.\end{split} \end{equation}

In particular, the modular form $I(\tau, L)$ is  a genus theta function.
\end{proposition}

\subsection{Shimura curve and Heegner divisors} \label{sect:Shimura}

In this subsection, we assume that $D>0$ has even number prime factors, and then $V= (V(D), \det)$ is of signature $(1, 2)$ and is anisotropic when $D>1$. According to \cite[Theorem 4.23]{KuIntegral},  the theta integral $I(g, \varphi)$ is a generating function of degrees of  some divisors with respect to the tautological line bundle over the Shimura curve associated to $V$. In this paper, these divisors are Heegner divisors in Shimura curves.

 Let  $H=\Gspin(V) $. It is know that there is an isomorphism $H\cong B^{\times}.$ The action acts on $V$ is explicit, $$g.v=gvg^{-1}, \quad g\in B^{\times} ,\quad v \in V.$$

There is
an  exact sequence
$$
1 \rightarrow \mathbb G_m \rightarrow H \rightarrow \SO(V) \rightarrow 1.
$$
Let  $\mathbb D$ be the Hermitian domain of oriented negative $2$-planes in $V(\R)$, and
$$
\mathcal L = \{ w \in V_\C=V(\C):\,  (w, w) =0,  (w, \bar w) <0\}.
$$
$H(\R)$ acts naturally on both.
The  map
$$
f: \mathcal L/\C^\times \cong  \mathbb D,  \quad w= u+ i v \mapsto  \R (-u) + \R v
$$
gives an $H(\R)$-equivariant isomorphism between $\mathcal L/\C^\times$ and $\mathbb D$. So $\mathcal L$ is a  (tautological) line bundle over $\mathbb D$. The Hermitian domain also has a tube representation which we need.  Indeed, we know
$$
\mathbb D=\H \cup \overline{\H}.
$$
Thus under this identification, the action of $\H(\R)$ on $\mathbb D$ becomes the usual linear fractional action.

Associated to  a compact open subgroup $K$ of $H(\hat \Q)$ is a Shimura curve $X_K$ over $\Q$ such that
$$
X_K(\C) = H(\Q) \backslash \mathbb D \times H(\hat\Q)/K.
$$
Moreover, $\mathcal L$ descends to a line bundle on $X_K$, which we continue to denote by $\mathcal L$. It can be identified with the line bundle of two variable modular forms of weight $(1, 1)$.   In this section, we always assume
$$
K = \hat{\OO}_D(N)^\times \subset H(\hat\Q)
$$
which preserves the lattice $L_D(N)$, where $\OO_D(N)$ is the Eichler order of conductor N. By the Strong Approximation theorem $H(\hat \Q)=H(\Q)K, $ one has $$X_K= X_0^D(N).$$

 Let $\Omega_0 = \frac{1}{2 \pi} y^{-2} dx\wedge dy$ be the differential form on $X_0^D(N)$, then
from \cite[(2.7)]{KRYComp} and
\cite[Lemma 5.3.2]{Miy} we know
\begin{align} \label{eq:volume}
\vol(X_0^D(N), \Omega) &:= \int_{X_0^D(N)} \Omega_0 = -2 [\OO_D^1 : \Gamma_0^D(N)] \zeta_D(-1)
\\
   &=\frac{ DN}6 \prod_{p|N} (1+p^{-1}) \prod_{p|D} (1-p^{-1} )  \in \frac{1}6 \Z  ,  \notag
\end{align}
where $\zeta_D(s) =\prod_{p\nmid D} (1-p^{-s})^{-1}$ is the partial zeta function, and $\OO_D$ is a maximal order of $B$ containing $\OO_D(N)$.

Recall the Kudla cycle on $X_K$.
Fix a $x \in V(\Q)$ with $\det (x) >0$ and $h \in H(\hat \Q)$, $x^\perp$
is a subspace of signature $(0, 2)$ and defines a sub-Shimura variety $Z(x)$ of $X_{h K h^{-1}}$, its right translate by $h$ gives a divisor $Z(x, h)$ in $X_K$.
Let $\varphi_f \in S(\hat V)^K$ and $m \in \Q_{>0}$. If there is a $x_0 \in V(\Q)$ such that $\det (x_0) =m$, we   define the associated Kudla cycle  $Z(m, \varphi_f)$  as
$$
Z(m, \varphi_f)= \sum_{j=1}^r  \varphi_f(h_j^{-1}x_0) Z(x_0, h_j),
$$
where
$$
\hbox{Supp}(\varphi_f) \cap \{ x \in V(\hat\Q):\,  \det x =m \} = \coprod_{j=1}^r  K h_j^{-1} x_0.
$$
Otherwise, we define $Z(m, \varphi_f) =0$.

 Let $L_m =\{ x \in L:\,  \det x =m\}$.
By the Strong Approximation theorem, one has
$H(\hat\Q) =H(\Q)K$.  So we have the decomposition with $Q(x_0) =m$,
$$
\hat L_m= \coprod K h_j^{-1} x_0,
$$
where $h_j \in  H(\Q)$. Then we know
$$
L_m = \coprod \Gamma_K h_j^{-1} x_0 = \coprod \Gamma_K x_j,  \quad x_j =h_j^{-1} x_0 \in L,
$$
where $\Gamma_K=K \cap H(\Q)$,
and
$$
Z(m, \varphi_f) = \sum_j Z(x, h_j) =\sum_j  Z(h_j^{-1} x_0) =\sum_j Z(x_j).
$$
where $Z(x_j)$ are Heegner points in this paper. Define $$\deg Z(m, \varphi_f) = \sum_{x\in  \Gamma_K  \setminus  L_{m}} \frac{1}  { \mid \Gamma_x \mid},$$ and $$Z_{D, N}(m):=Z(m, \cha (\widehat{L_D(N)}),$$ where $\Gamma_x$  is the stablizer subgroup of $x$  in $\Gamma_K$. Let $$
r_{D, N}(m) = \frac{\deg Z_{D, N}(m)}{ \vol(X_0^D(N), \Omega)}
$$
be as in the introduction.

\section{main results}\label{result}
In Section \ref{sect:definite and indefinite}, we introduced the number $r_{D, N}$ for both the definite and indefinite space $V(D)$. This number is related to the coefficient of theta integral for both spaces as follows:

\begin{proposition}\label{pro 5.1}
Let $D>1$, for quadratic space $V(D)$, let $\varphi_f =\cha (\widehat{L_D(N)})$, one has
$$
I(\tau, \varphi_f \varphi_\infty^{sp} ) = v^{-\frac{3}{4}} I(g_\tau^{\prime}, \varphi_f\varphi_\infty^{sp}) = \sum_{m=0}^\infty r_{D, N}(m) q^m,
$$
where  $r_{D, N}(0) =1$, and for $m >0$.
\end{proposition}

\begin{proof} When $V(D)$ is definite, this is  Propostion \ref{pro4.1}. When $V(D)$ is indefinite, write
$$
I(\tau, \varphi_f \varphi_\infty^{sp} )  = \sum_{m=0}^\infty c(m) q^m.
$$
By \cite[Section 4.8]{KuIntegral} and \cite[Theorem 3.4]{Fu}, one has $c(0) =1$ and for $m >0$,
$$c(m)=\frac{\deg Z_{D, N}(m)}{\vol(X_K, \Omega)}.$$ So $c(m) = r_{D, N}(m)$ as claimed.
\end{proof}
{\bf Proof of Theorem \ref{theo1.1} and Corollary \ref{coro1.2} } : From Proposition \ref{pro3.5} and the Proposition \ref{pro 5.1} one proves Theorem \ref{theo1.1}. From this theorem, the definition of normalized degree of Heegner divisors(\ref{hdeg}) and the volume formula (\ref{eq:volume}) one obtain the corollary easily.\\

{\bf Proof of Theorem \ref{theo1.3} }:  Let $V^{(1)} =V(D)$ and $V^{(2)} =V(Dp)$ as above, and let $\varphi^{(i)} =\prod_l \varphi_l^{(i)} \in S(V^{(i)}(\A)$ be as follows. For $l \nmid p\infty$, we identify $L_D(N)_l$ with $L_{Dp}(N)_l$ and denote $\varphi_l^{(i)} = \cha(L_D(N)_l)$. If $D$ has odd number of primes, let
$$
\varphi_\infty^{(1)} =\varphi_\infty^{ra},  \quad \varphi_\infty^{(2)} =\varphi_\infty^{sp}.
$$
Otherwise, $D$ has even number of primes, let $$
\varphi_\infty^{(1)} =\varphi_\infty^{sp},  \quad \varphi_\infty^{(2)} =\varphi_\infty^{ra}.
$$
Finally, let
$$
\varphi_p^{(1)} = -\frac{2}{p-1} \varphi_{p}^{sp} + \frac{p+1}{p-1}\varphi_{p, 1}^{sp}, \quad \varphi_p^{(2)} = \varphi_{p}^{ra}.
$$
Then  $\varphi^{(1)}$ and $\varphi^{(2)}$ match by the results in Section \ref{sect:matching}. So one has  by Proposition  \ref{globalmatch}
$$
I(\tau, \varphi^{(1)}) =I(\tau, \varphi^{(2)}).
$$
Comparing $m$-th coefficients of the both sides, and from Proposition \ref{pro 5.1} and Proposition \ref{pro4.1} one proves Theorem \ref{theo1.3}.\\

\section*{Acknowledgements}
This paper was inspired by Kudla's matching principle, thanks him for his influence. I also thank Tonghai Yang for his good suggestion and useful discussion. The author is grateful to the Mathematical Science Center of Tsinghua University, for providing him a good opportunity to visit and a good research environment in the summer 2013. The author is partially supported by NSFC (Nos. 11171141),
 NSFJ (Nos. BK2010007), PAPD and the Cultivation Fund of the Key
Scientific and Technical Innovation Project, Ministry of Education
of China (No.708044), NSFC (11326052).

\end{document}